\documentclass[eqno]{amsart}
\usepackage{latexsym, amssymb, amsmath}
\begin{document}

\def\theequation{\thesection.\arabic{equation}}
  \makeatletter
  \@addtoreset{equation}{section}
  \makeatother

\newcommand{\olapla}{\overline\Delta}
\newcommand{\onabla}{\overline\nabla}
\newcommand{\wnabla}{\widetilde\nabla}
\newcommand{\lapla}{\Delta}
\newcommand{\nablaf}{\nabla \hspace{-9pt} \nabla \hspace{-7.6pt} \nabla}
\newcommand{\wnablaf}{\widetilde\nabla \hspace{-9pt} \nabla \hspace{-7.6pt} \nabla}
\newcommand{\onablaf}{\overline\nabla \hspace{-9pt} \nabla \hspace{-7.6pt} \nabla}
\newcommand{\fpt}{\frac{\partial}{\partial t}}
\newcommand{\p}{\phi}
\newcommand{\ka}{\kappa}
\newcommand{\g}{\gamma}
\newcommand{\vp}{\varphi}
\newcommand{\pmton}{\p:(M,g)\rightarrow(N,h)}
\newcommand{\met}{\langle \cdot , \cdot \rangle}
\newcommand{\la}{\langle}
\newcommand{\ra}{\rangle}

\newtheorem{conj}{Conjecture}
\newtheorem{thm}{Theorem}[section]
\newtheorem{lem}[thm]{Lemma}
\newtheorem{cor}[thm]{Corollary}
\newtheorem{prop}[thm]{Proposition}
\newtheorem{rem}[thm]{Remark}
\theoremstyle{definition}
\newtheorem{defn}[thm]{Definition}
\newtheorem*{thmA}{Theorem~A}
\newtheorem*{thmB}{Theorem~B}
\newtheorem*{ex}{Example}
\newtheorem{prob}{Problem}

  \title{Triharmonic isometric immersions into a manifold of non-positively constant curvature}
  \title[Triharmonic isometric immersions]
  {Triharmonic isometric immersions into a manifold of non-positively constant curvature}
  \author{Shun Maeta}
  \address{Shumei University, Yachiyo, 276-0003, Japan}
  \email{shun.maeta@@gmail.com}
\author{Nobumitsu Nakauchi}
  \address{Graduate School of Science and Engineering, \newline
  Yamaguchi University, 
  Yamaguchi, 753-8512, Japan}
  \email{nakauchi@@yamaguchi-u.ac.jp}
   \author{Hajime Urakawa}
  \address{Division of Mathematics, Graduate School of Information Sciences, Tohoku University, Aoba 6-3-09, Sendai, 980-8579, Japan}
  \curraddr{Institute for International Education, 
  Tohoku University, Kawauchi 41, Sendai 980-8576, Japan}
  \email{urakawa@@math.is.tohoku.ac.jp}
  \thanks{
  Supported by the Grant-in-Aid for the Scientific Research, (C) 
  No 24540213, No. 25400154, and the Grant-in-Aid for Research Activity Start-up, No. 25887044, 
  Japan Society for the Promotion of Science. 
  }
\maketitle

\begin{abstract}
 A triharmonic map is a critical point of the 3-energy in the space of smooth maps between two Riemannian manifold.  
  We study the generalized Chen's conjecture for a triharmonic isometric immersion $\varphi$ into a space form of non-positively constant curvature. We show that if the domain is complete and both 
the 4-energy of $\varphi$, and the $L^4$-norm of the tension field $\tau(\varphi)$, 
are finite, then such an immersion $\varphi$ is minimal.
 \end{abstract}


\qquad\\

\section{Introduction}
Harmonic maps play a central role in geometry;\,they are critical points of the energy functional 
$E(\varphi)=\frac12\int_M\vert d\varphi\vert^2\,v_g$ 
for smooth maps $\varphi$ of $(M,g)$ into $(N,h)$. The Euler-Lagrange equations are given by the vanishing of the tension filed 
$\tau(\varphi)$. 
In 1983, J. Eells and L. Lemaire \cite{EL2} extended the notion of harmonic map to  
$k$-harmonic map, which are, 
by definition, 
critical points of the $k$-energy functional
\begin{equation}
E_k(\varphi)=\frac12\int_M
\vert(d+\delta)^k\varphi\vert^2\,v_g.
\end{equation}
For $k=1$, $E_1=E$, and for $k=2$, 
after G.Y. Jiang \cite{J} studied the first and second variation formulas of 
$E_2$ $(k=2)$, 
extensive studies in this area have been done
(for instance, see 
 \cite{II}, \cite{CMP}, \cite{LO}, \cite{O1}, \cite{S1}, \cite{MO1},  
 \cite{LO2}, \cite{BFO}, \cite{IIU2}, \cite{IIU}, 
 etc.). Notice that harmonic maps are always biharmonic by definition. 
\par
For harmonic maps, it is well known that: 
\par
{\em If a domain manifold $(M,g)$ is complete and has non-negative Ricci curvature, and the sectional curvature of a target manifold $(N,h)$ is non-positive, then 
every energy finite harmonic 
map is a constant map} (cf. \cite{SY}). 
\par
Therefore, it is a natural question to consider 
$k$-harmonic maps into a Riemannian manifold of non-positive curvature. 
In this connection, Baird, Fardoun and Ouakkas (cf. \cite{BFO}) showed that:  
\par
{\em If a non-compact Riemannian manifold 
$(M,g)$ is complete and has non-negative Ricci curvature and $(N,h)$ has non-positive sectional curvature, then every bienergy finite biharmonic map of $(M,g)$ into 
$(N,h)$ is harmonic}.  
\vskip0.3cm\par
In our previous paper (\cite{NU3}), we showed without the Ricci curvature condition of $(M,g)$, that 
\begin{thm} 
Under only the assumptions of completeness of 
$(M,g)$ and non-positivity of curvature of $(N,h)$,  \par
$(1)$ every biharmonic map 
$\varphi:\,(M,g)\rightarrow (N,h)$ with 
finite energy and 
finite bienergy 
must be harmonic. 
\par
$(2)$ In the case ${\rm Vol}(M,g)=\infty$, 
every biharmonic map 
$\varphi:\,(M,g)\rightarrow (N,h)$ with finite bienergy 
is harmonic. 
\end{thm}
\vskip0.3cm\par
\par
We obtained 
(cf. \cite{NU1}, \cite{NU2}, \cite{NU3}) 
\begin{thm} 
Assume that $(M,g)$ is a complete 
Riemannian manifold, and let 
$\varphi:\,(M,g)\rightarrow (N,h)$ 
is an isometric immersion, and the sectional curvature of $(N,h)$ is non-positive. 
If $\varphi:\,(M,g)\rightarrow (N,h)$ is biharmonic 
and $\int_M\vert{\bf H}\vert^2\,v_g<\infty$, then 
it is minimal. 
Here, ${\bf H}$ is the mean curvature normal vector field of the isometric immersion $\varphi$. 
\end{thm}
\vskip0.3cm\par
The above theorems 
gave an affirmative answer to the generalized B.Y. Chen's conjecture (cf. \cite{CMP})
under natural conditions:

\begin{center}
\lq\lq Every biharmonic isometric immersion into a Riemannian manifold\\ of non-positive curvature must be harmonic.\rq\rq
\end{center}
\vskip0.2cm\par

Now in this paper, we will discuss the $3$-energy $E_3$, ($k=3$) and 
a {\em triharmonic} map which is a critical point of the $3$-energy $E_3$ 
in the space of smooth maps of $M$ into $N$. 
We first show (cf. Theorem 2.1) the first variational formula of triharmonic maps which is of simple form in the case of 
an isometric immersion into the Riemannian manifold 
of constant curvature (cf. Corollary 2.3). 
Then, we want to show that the generalized Chen's conjecture is true for a triharmonic isometric immersion into a Riemannian manifold of non-positive curvature. More precisely, we will show that 
\begin{thm}[(cf. Theorem 2.4 and \ref{sub Th})]
Assume that $\varphi:\,(M,g)\rightarrow N(c)$ is an isometric immersion of a complete Riemannian manifold $(M,g)$ into 
a Riemannian manifold $N(c)$ of non-positively constant curvature $c$. 
In the case $c<0$, 
if $\varphi$ is triharmonic and both the extended $4$-energy $\widetilde{E}_4(\varphi)=\frac12\int_M\vert \overline{\Delta}\tau(\varphi)\vert^2\,v_g$ and the $L^4$-norm $\int_M\vert \tau(\varphi)\vert^4\,v_g$ are finite, 
then $\varphi$ 
is minimal. 
\par
In the case $c=0$, the same conclusion holds if we assume more
$E_2(\varphi)=\frac12\int_M\vert \tau(\varphi)\vert^2\,v_g<\infty$ or $E_3(\varphi)=\frac12\int_M\vert\overline{\nabla}\tau(\varphi)\vert^2\,v_g<\infty$.  
\end{thm}
\vskip0.6cm\par
\section{Preliminaries and statement of main theorem}
In this section, we prepare materials for the first and second variational formulas for the bienergy functional and biharmonic maps. 
Let us recall the definition of a harmonic map $\varphi:\,(M,g)\rightarrow (N,h)$, of a compact Riemannian manifold $(M,g)$ into another Riemannian manifold $(N,h)$, 
which is an extremal 
of the {\em energy functional} defined by 
$$
E(\varphi)=\int_Me(\varphi)\,v_g, 
$$
where $e(\varphi):=\frac12\vert d\varphi\vert^2$ is called the energy density 
of $\varphi$.  
That is, for any variation $\{\varphi_t\}$ of $\varphi$ with 
$\varphi_0=\varphi$, 
\begin{equation}
\frac{d}{dt}\bigg\vert_{t=0}E(\varphi_t)=-\int_Mh(\tau(\varphi),V)v_g=0,
\end{equation}
where $V\in \Gamma(\varphi^{-1}TN)$ is a variation vector field along $\varphi$ which is given by 
$V(x)=\frac{d}{dt}\vert_{t=0}\varphi_t(x)\in T_{\varphi(x)}N$, 
$(x\in M)$, 
and  the {\em tension field} is given by 
$\tau(\varphi)
=\sum_{i=1}^mB(\varphi)(e_i,e_i)\in \Gamma(\varphi^{-1}TN)$, 
where 
$\{e_i\}_{i=1}^m$ is a locally defined frame field on $(M,g)$, 
and $B(\varphi)$ is the second fundamental form of $\varphi$ 
defined by 
\begin{align}
B(\varphi)(X,Y)&=(\widetilde{\nabla}d\varphi)(X,Y)\nonumber\\
&=(\widetilde{\nabla}_Xd\varphi)(Y)\nonumber\\
&=\overline{\nabla}_X(d\varphi(Y))-d\varphi(\nabla_XY),
\end{align}
for all vector fields $X, Y\in {\frak X}(M)$. 
Here, 
$\nabla$, and
$\nabla^N$, 
 are Levi-Civita connections on $TM$, $TN$  of $(M,g)$, $(N,h)$, respectively, and 
$\overline{\nabla}$, and $\widetilde{\nabla}$ are the induced ones on $\varphi^{-1}TN$, and $T^{\ast}M\otimes \varphi^{-1}TN$, respectively. By (2.1), $\varphi$ is harmonic if and only if $\tau(\varphi)=0$. 
\par
The second variation formula is given as follows. Assume that 
$\varphi$ is harmonic. 
Then, 
\begin{equation}
\frac{d^2}{dt^2}\bigg\vert_{t=0}E(\varphi_t)
=\int_Mh(J(V),V)v_g, 
\end{equation}
where 
$J$ is an elliptic differential operator, called the
{\em Jacobi operator}  acting on 
$\Gamma(\varphi^{-1}TN)$ given by 
\begin{equation}
J(V)=\overline{\Delta}V-{\mathcal R}(V),
\end{equation}
where 
$\overline{\Delta}V=\overline{\nabla}^{\ast}\overline{\nabla}V
=-\sum_{i=1}^m\{
\overline{\nabla}_{e_i}\overline{\nabla}_{e_i}V-\overline{\nabla}_{\nabla_{e_i}e_i}V
\}$ 
is the {\em rough Laplacian} and 
${\mathcal R}$ is a linear operator on $\Gamma(\varphi^{-1}TN)$
given by 
${\mathcal R}(V)=
\sum_{i=1}^mR^N(V,d\varphi(e_i))d\varphi(e_i)$,
and $R^N$ is the curvature tensor of $(N,h)$ given by 
$R^N(U,V)=\nabla^N{}_U\nabla^N{}_V-\nabla^N{}_V\nabla^N{}_U-\nabla^N{}_{[U,V]}$ for $U,\,V\in {\frak X}(N)$.   
\par
J. Eells and L. Lemaire \cite{EL2} proposed polyharmonic ($k$-harmonic) maps 
and 
Jiang \cite{J} studied the first and second variation formulas of biharmonic maps. 
Let us consider the {\em $k$-energy} defined by 
\begin{align}
E_k(\varphi)&=\frac12\int_M\vert (d+\delta)^k\,\varphi\vert^2\,v_g
\end{align}
for a smooth map $\varphi$ from $M$ into $N$, and 
$\varphi$ is called {\em $k$-harmonic} if it is a critical point of 
$E_k$ $(k=1,2,3,\cdots)$.  
Here, we also define  
the {\em extended $k$-energy} $\widetilde{E}_k$ is given (cf. \cite{IIU}, p.270) as follows:
\begin{align}
\widetilde{E}_k(\varphi)=\left\{
\begin{aligned}
&\int_M\vert W^{\ell}_{\varphi}\vert^2\,v_g\qquad\quad (k=2\ell),\\
&\int_M\vert\overline{\nabla}W^{\ell}_{\varphi}\vert^2\,v_g\qquad(k=2\ell+1)
\end{aligned}
\right.
\end{align}
where $W^{\ell}_{\varphi}$ is given by 
\begin{align}
W^{\ell}_{\varphi}=\underbrace{\overline{\Delta}\cdots\overline{\Delta}}_{\ell-1}\tau(\varphi)
\end{align} 
if $\ell\geq 1$. 
Notice that 
$E_k(\varphi)=\widetilde{E}_k(\varphi)$ $(k=1,2,3)$, but 
for $k=4$, 
it holds that 
\begin{align}
E_4(\varphi)&=\frac12\int_M\vert (d+\delta)(d+\delta)\tau(\varphi)\vert^2\,v_g
\nonumber\\
&=\frac12\int_M\vert d\,d\,\tau(\varphi)\vert^2\,v_g+\frac12\int_M\vert\overline{\Delta}\tau(\varphi)\vert^2\,v_g\nonumber\\
&=\frac12\int_M\vert d\,d\,\tau(\varphi)\vert^2\,v_g+\widetilde{E}_4(\varphi).
\end{align}
If $\ell=0$, we put $W^0_{\varphi}=\varphi$ and 
$E_1(\varphi)=\frac12\int_M\vert\overline{\nabla}\varphi\vert^2\,v_g=\frac12\int_M\vert d\varphi\vert^2\,v_g$.  
For $k=1$, $E_1=\widetilde{E}_1=E$, 
and for $k=2$,   
the {\em bienergy functional} $E_2$ 
is given by 
\begin{equation}
E_2(\varphi)=\frac12\int_M\vert\tau(\varphi)\vert ^2v_g, 
\end{equation}
where 
$\vert V\vert^2=h(V,V)$, $V\in \Gamma(\varphi^{-1}TN)$.  
Then, the first variation formula of the bienergy functional 
is given by
\begin{equation}
\frac{d}{dt}\bigg\vert_{t=0}E_2(\varphi_t)
=-\int_Mh(\tau_2(\varphi),V)v_g.
\end{equation}
Here, 
\begin{equation}
\tau_2(\varphi)
:=J(\tau(\varphi))=\overline{\Delta}(\tau(\varphi))-{\mathcal R}(\tau(\varphi)),
\end{equation}
which is called the {\em bitension field} of $\varphi$, and 
$J$ is given in $(2.4)$.  
A smooth map $\varphi$ of $(M,g)$ into $(N,h)$ is said to be 
{\em biharmonic} if 
$\tau_2(\varphi)=0$. 
\par
\par
For $k=3$, the first variation formula of the {\em trienergy} $E_3$ given by 
\begin{equation}
E_3(\psi)=\frac12\int_M\vert (d+\delta)(d+\delta)(d+\delta)\psi\vert^2\,v_g
\end{equation} 
is given as follows: 
\begin{thm}  The first variational formula of $E_3$ is given by 
\begin{align}
&\frac{d}{dt}\bigg\vert_{t=0}E_3(\varphi_t)=-\int_M\langle \tau_3(\varphi),V\rangle\,v_g,
\\
&\tau_3(\varphi)=J(\overline{\Delta}(\tau(\varphi)))-\sum_{i=1}^m
R^N(\overline{\nabla}_{e_i}\tau(\varphi),\tau(\varphi))d\varphi(e_i).
\end{align}
Here, $\tau_3(\varphi)$ is called the {\em tritension field} of $\varphi$.  
\end{thm}
For completeness, we give a proof. The proof is standard.
\par
For $V\in \Gamma(\varphi^{-1}TN)$, let 
$\varphi_t$ $(-\epsilon<t<\epsilon)$ of $\varphi$ be 
a $C^{\infty}$ variation of $V$ with $\varphi_0=\varphi$, 
$V(x)=\frac{d}{dt}\big\vert_{t=0}\varphi_t(x)$ $(x\in M)$. 
Let us define a $C^{\infty}$ map $F:\,(-\epsilon,\epsilon)\times M\rightarrow N$, in such a way that 
\begin{equation}
\left\{
\begin{aligned}
F(0,x)&=\varphi(x),\qquad x\in M,\nonumber\\
F(t,x)&=\varphi_t(x),\qquad -\epsilon<t<\epsilon,\,\,x\in M.
\end{aligned}
\right.
\end{equation}
We need the following lemma. 
\begin{lem} 
For every smooth vector field $X$ in $M$, 
\begin{align}
\overline{\nabla}_{\frac{\partial}{\partial t}}
\overline{\nabla}_X
\tau(F)
\bigg\vert_{t=0}
&=
-\overline{\nabla}_X(\overline{\Delta} V)
+\sum_{j=1}^m\overline{\nabla}_X
\left(
R^N(V,d\varphi(e_j))d\varphi(e_j)
\right)\nonumber\\
&\qquad+R^N(V,d\varphi(X))\tau(\varphi),
\end{align}
where $\{e_j\}_{j=1}^m$ is a locally defined orthonormal frame field on $(M,g)$. 
\end{lem}
\begin{proof}
Since $[\frac{\partial}{\partial t},X]=0$, we have 
\begin{align}
\overline{\nabla}_{\frac{\partial}{\partial t}}\big(
\overline{\nabla}_X\tau(F)
\big)
=\overline{\nabla}_X\big(
\overline{\nabla}_{\frac{\partial}{\partial t}}\tau(F)\big)
+R^N\big(dF\big(
\frac{\partial}{\partial t}\big),dF(X)
\big)\,\tau(F).
\end{align}
Due to (23) and (22) in Jiang's paper (\cite{J}, English version, p. 214),  we have 
\begin{align}
\overline{\nabla}_{\frac{\partial}{\partial t}}\tau(F)\bigg\vert_{t=0}
&=\overline{\nabla}_{\frac{\partial}{\partial t}}\big(
\sum_{j=1}^m
(\widetilde{\nabla}_{e_j}dF)(e_j)
\big)\bigg\vert_{t=0}\nonumber\\
&=-\overline{\Delta}V+\sum_{j=1}^mR^N(V,d\varphi(e_j))d\varphi(e_j). 
\end{align}
By substituting (2.17) into (2.16), we have (2.15). 
\end{proof}
\vskip0.6cm\par
\begin{proof}[Proof of Theorem 2.1] 
By definition of $E_3$, we have for every $C^{\infty}$ map 
$\varphi:\,M\rightarrow N$, 
\begin{align}
E_3(\varphi)&=\frac12\int_M\langle d(\delta d\varphi),d(\delta d\varphi)\rangle\,v_g\nonumber\\
&=\frac12\int_M\sum_{i=1}^m\langle 
\overline{\nabla}_{e_i}(\tau(\varphi)), \overline{\nabla}_{e_i}(\tau(\varphi))\rangle\,v_g.
\end{align}
By Lemma 2.2, we have for a $C^{\infty}$ variation $\varphi_t$ of $V$ with $\varphi_0=\varphi$, 
\begin{align}
\frac{d}{dt}\bigg\vert_{t=0}E_3(\varphi_t)&=
\int_M\sum_{i=1}^m\langle
\overline{\nabla}_{\frac{\partial}{\partial t}}(\overline{\nabla}_{e_i}\tau(F)),\overline{\nabla}_{e_i}\tau(F)\rangle\,v_g\bigg\vert_{t=0}\nonumber\\
&=\int_M\sum_{i=1}^m
\bigg\langle
-\overline{\nabla}_{e_i}(\overline{\Delta}V)
+\sum_{j=1}^m\overline{\nabla}_{e_i}\big(
R^N(V,d\varphi(e_j))d\varphi(e_j)\big)\nonumber\\
&\qquad +R^N(V,d\varphi(e_i))\tau(\varphi),\overline{\nabla}_{e_i}(\tau(\varphi))
\bigg\rangle\,v_g.
\end{align}
Here, by using formula 
for every $\omega_j\in \Gamma(\varphi^{-1}TN)$, $(j=1,2)$, 
\begin{align}
\int_M\sum_{i=1}^m\langle 
\overline{\nabla}_{e_i}\omega_1,\overline{\nabla}_{e_i}\omega_2\rangle\,v_g 
=\int_M\langle\overline{\Delta}\omega_1,\omega_2\rangle\,v_g, 
\end{align}
we have 
\begin{align}
\frac{d}{dt}\bigg\vert_{t=0}E_3(\varphi_t)&=\int_M
\langle V,-\overline{\Delta}^2\tau(\varphi)\rangle\,v_g\nonumber\\
&\quad+\int_M\sum_{j=1}^m\langle 
R^N(V,d\varphi(e_j))d\varphi(e_j),\overline{\Delta}\tau(\varphi)\rangle\,v_g
\nonumber\\
&\quad
+\int_M\sum_{j=1}^m\langle R^N(V,d\varphi(e_j))\tau(\varphi),\overline{\nabla}_{e_j}\tau(\varphi)\rangle\,v_g\nonumber\\
&=
\int_M\bigg\langle V, 
-\overline{\Delta}^2\tau(\varphi)
+\sum_{j=1}^mR^N(\overline{\Delta}\tau(\varphi),d\varphi(e_j))d\varphi(e_j)\nonumber\\
&\qquad\qquad
+\sum_{j=1}^mR^N(\overline{\nabla}_{e_j}\tau(\varphi),\tau(\varphi))d\varphi(e_j)
\bigg\rangle\,v_g,\nonumber
\end{align}
in which we used 
the property $\langle R^N(v_3,v_4)v_2,v_1\rangle=\langle R^N(v_1,v_2)v_4,v_3\rangle$. 
We have Theorem 2.1. 
\end{proof}
\vskip0.6cm\par
\begin{cor} Assume that $(N,h)$ is an $n$-dimensional 
Riemannian manifold $(N(c), h)$ of constant curvature $c$,  and 
let $\varphi:\,(M,g)\rightarrow N(c)$ be an isometric immersion. Then, we have 
\begin{align}
\tau_3(\varphi)&=
\overline{\Delta}^2\tau(\varphi)-\sum_{j=1}^mR^N(\overline{\Delta}\tau(\varphi),d\varphi(e_j))d\varphi(e_j)\nonumber\\
&\qquad\qquad\,\, -c\,h(\tau(\varphi),\tau(\varphi))\,\tau(\varphi).
\end{align} 
\end{cor}
\begin{proof}
Since $R^N(X,Y)Z=c\,\{h(Y,Z)X-h(X,Z)\,Y\}$, $(X,Y,Z\in {\frak X}(N))$, we have 
\begin{align}
\sum_{j=1}^mR^N(\overline{\nabla}_{e_j}\tau(\varphi),\tau(\varphi))d\varphi(e_j)
&=\sum_{j=1}^mc\,\{
h(\tau(\varphi),d\varphi(e_j))\,\overline{\nabla}_{e_j}\tau(\varphi)\nonumber\\
&\qquad 
-h(\overline{\nabla}_{e_j}\tau(\varphi),d\varphi(e_j))\,\tau(\varphi)
\}\nonumber\\
&=-c\sum_{j=1}^mh(\overline{\nabla}_{e_j}\tau(\varphi),d\varphi(e_j))\,\tau(\varphi).
\end{align}
Because the tension field $\tau(\varphi)$ is orthogonal to the subspace 
$d\varphi(T_xM)$ $(x\in M)$
since $\varphi:\,(M,g)\rightarrow N(c)$ is an isometric immersion. 
Then, 
\begin{align}
\sum_{j=1}^mh(\overline{\nabla}_{e_j}\tau(\varphi),d\varphi(e_j))
&=\sum_{j=1}^m \{
e_j(h(\tau(\varphi),d\varphi(e_j)))-h(\tau(\varphi),\overline{\nabla}_{e_j}(d\varphi(e_j)))\}\nonumber\\
&=-h(\tau(\varphi),\sum_{j=1}^m\overline{\nabla}e_j(d\varphi(e_j)))\nonumber\\
&=-h(\tau(\varphi),\tau(\varphi)+\sum_{j=1}^md\varphi(\nabla_{e_j}e_j))\nonumber\\
&=-h(\tau(\varphi),\tau(\varphi))
\end{align} 
since $\tau(\varphi)=\sum_{j=1}^m\{\overline{\nabla}_{e_j}(d\varphi(e_j))-d\varphi(\nabla_{e_j}e_j)\}$. By substituting (2.23) into (2.22), 
(2.22) is equal to 
$c\,h(\tau(\varphi),\tau(\varphi))\,\tau(\varphi)$. 
Then, the right hand side of (2.14) is equal to 
$J(\overline{\Delta}(\tau(\varphi)))-c\,h(\tau(\varphi),\tau(\varphi))\,\tau(\varphi)$. 
We obtain Corollary 2.3. 
\end{proof}
\vskip0.6cm\par
Then, we can state our main theorem. 
\begin{thm}\label{main Th}
Let $\varphi:\,(M,g)\rightarrow N(c)$ be an isometric immersion of a complete Riemannian manifold $(M,g)$ into 
a Riemannian manifold $N(c)$ of non-positively constant curvature $c$. 
\par
$(1)$ In the case of $c<0$,  
if $\varphi$ is triharmonic and both the extended 
$4$-energy $\widetilde{E}_4(\varphi)=\frac12\int_M\vert \overline{\Delta}\tau(\varphi)\vert^2\,v_g$ and the $L^4$-norm $\int_M\vert \tau(\varphi)\vert^4\,v_g$ are finite, 
then $\varphi$ 
is harmonic, i.e., minimal. 
\par
$(2)$ In the case of $c=0$, 
and $\mbox{\rm Vol}(M,g)=\infty$, 
if $\varphi$ is triharmonic, and $\widetilde{E}_4(\varphi)=\frac12\int_M\vert \overline{\Delta}\tau(\varphi)\vert^2\,v_g<\infty$, $E_2(\varphi)=\frac12\int_M\vert\tau(\varphi)\vert^2\,v_g<\infty$ and 
$\int_M\vert \tau(\varphi)\vert^4\,v_g<\infty$, 
then $\varphi$ 
is harmonic, i.e., minimal. 
\par
$(3)$ 
In the case of $c=0$ and $\mbox{\rm Vol}(M,g)<\infty$, 
if $\varphi$ is triharmonic, and $\widetilde{E}_4(\varphi)=\frac12\int_M\vert \overline{\Delta}\tau(\varphi)\vert^2\,v_g<\infty$, 
$E_3(\varphi)=\frac12\int_M\vert\overline{\nabla}\tau(\varphi)\vert^2\,v_g<\infty$ and 
$\int_M\vert \tau(\varphi)\vert^4\,v_g<\infty$, 
then $\varphi$ 
is harmonic, i.e., minimal. 
\end{thm}
\vskip0.6cm\par
\section{Proof of Theorem 2.4}
In this section, we will give a proof of 
Theorem 2.4 which consists of 
eight steps. 

\begin{proof}[Proof of Theorem 2.4]
\par
({\it The first step}) \quad 
For a fixed point $x_0\in M$, and 
for every 
$0<r<\infty$, 
we first take a cut-off  $C^{\infty}$ function $\eta$ on $M$ 
(for instance, see \cite{K}) satisfying that 
\begin{equation}
\left\{
\begin{aligned}
0\leq &\eta(x)\leq 1\quad (x\in M),\\
\eta(x)&=1\qquad\quad (x\in B_r(x_0)),\\
\eta(x)&=0\qquad\quad (x\not\in B_{2r}(x_0)),\\
\vert\nabla\eta\vert&\leq\frac{2}{r}
\qquad\,\,\, (x\in M).
\end{aligned}
\right.
\end{equation}
\par
\vskip0.6cm\par
For a triharmonic map
$\varphi:\,(M,g)\rightarrow N(c)$, 
the tritension field is given as 
\begin{align}
\tau_3(\varphi)&=
\overline{\Delta}^2(\tau(\varphi))
-\sum_{i=1}^m
R^N(\overline{\Delta}\tau(\varphi),d\varphi(e_i))d\varphi(e_i)\nonumber\\
&\qquad\qquad\quad\,\, 
-c\,h(\tau(\varphi),\tau(\varphi))\,\tau(\varphi)\nonumber\\
&=0,
\end{align}
where $\overline{\Delta}=-\sum_{i=1}^m\{\overline{\nabla}_{e_i}\overline{\nabla}_{e_i}-\overline{\nabla}_{\nabla_{e_i}e_i}\}$. 
By taking the inner product of (3.2) and $\overline{\Delta}\tau(\varphi)\,\eta^2$ 
and integrate over $M$, so we have 
\begin{align}
&\int_M\langle\overline{\Delta}^2(\tau(\varphi)),\,\overline{\Delta}\tau(\varphi)\,\eta^2\rangle\,v_g\nonumber\\
&\quad-\int_M
\sum_{i=1}^m
\langle R^N(\overline{\Delta}\tau(\varphi),d\varphi(e_i))d\varphi(e_i),\overline{\Delta}\tau(\varphi)\rangle\,\eta^2\,v_g\nonumber\\
&\quad-c\,\int_M\langle\tau(\varphi),\tau(\varphi)\rangle\,\langle \tau(\varphi),\overline{\Delta}\tau(\varphi)\rangle\,\eta^2\,v_g\nonumber\\
&\quad= 0.
\end{align}
\vskip0.6cm\par
({\it The second step})\quad For the first term of the left hand side of (3.3),  
\begin{align}
\int_M\langle\overline{\Delta}^2(\tau(\varphi)),\,\overline{\Delta}\tau(\varphi)\,\eta^2\rangle\,v_g
&=\int_M\sum_{i=1}^m\langle 
\overline{\nabla}_{e_i}(\overline{\Delta}\tau(\varphi)),\overline{\nabla}_{e_i}(\overline{\Delta}\tau(\varphi)\,\eta^2)\rangle\,v_g.
\end{align}
Here by using that $e_i(\eta^2)=2\eta\,e_i(\eta)$ and 
$$
\overline{\nabla}_{e_i}(\overline{\Delta}\tau(\varphi)\,\eta^2)=
\overline{\nabla}_{e_i}(\overline{\Delta}\tau(\varphi))\,\eta^2+2\eta\,\nabla_{e_i}\eta\,\overline{\Delta}\tau(\varphi),
$$
(3.4) coincides with 
\begin{align}
\int_M
&\vert\overline{\nabla}\,\overline{\Delta}\tau(\varphi)\vert^2\,\eta^2\,v_g\nonumber\\
&+2\int_M\sum_{i=1}^m\langle 
\eta\,\overline{\nabla}_{e_i}(\overline{\Delta}\tau(\varphi)),\nabla_{e_i}\eta\,\overline{\Delta}\tau(\varphi)\rangle\,v_g,
\end{align}
where we used 
$$
\vert\overline{\nabla}\,\overline{\Delta}\tau(\varphi)\vert^2=\sum_{i=1}^m 
\langle\overline{\nabla}_{e_i}(\overline{\Delta}\tau(\varphi)),\overline{\nabla}_{e_i}(\overline{\Delta}\tau(\varphi))\rangle.
$$
For the second term of (3.5), 
put 
$S_i:=\eta\,\overline{\nabla}_{e_i}(\overline{\Delta}\tau(\varphi))$, and 
$T_i:=\nabla_{e_i}\eta\,\overline{\Delta}\tau(\varphi)$ ($i=1\,\cdots,m$), and 
recall the Young's inequality: 
for every $\epsilon>0$, 
$$
\pm2\langle S_i,T_i\rangle\leq \epsilon\,\vert S_i\vert^2+\frac{1}{\epsilon}\,\vert T_i\vert^2,
$$    
because of the inequality 
 $
 0\leq \vert \sqrt{\epsilon}\,S_i\pm\frac{1}{\sqrt{\epsilon}}\,T_i\vert^2.  
 $
 Therefore, (3.5) is bigger than or equal to
 \begin{align}
 \int_M&\vert\overline{\nabla}\,\overline{\Delta}\tau(\varphi)\vert^2\,\eta^2\,v_g\nonumber\\
 &\qquad\qquad-\bigg\{
 \epsilon\,\int_M\eta^2\,\vert \overline{\nabla}\,\overline{\Delta}\tau(\varphi)\vert^2\,v_g+\frac{1}{\epsilon}\,\int_M 
 \vert\nabla\eta\vert^2\,\vert\overline{\Delta}\tau(\varphi)\vert^2\,v_g
 \bigg\}\nonumber\\
 &=(1-\epsilon)
 \int_M\vert\overline{\nabla}\,\overline{\Delta}\tau(\varphi)\vert^2\,\eta^2\,v_g
 -\frac{1}{\epsilon}\int_M 
 \vert\overline{\Delta}\tau(\varphi)\vert^2\,\vert\nabla\eta\vert^2\,v_g.
 \end{align}
\vskip0.6cm\par
({\it The third step})\quad For the second term of the left hand side of (3.3), 
\begin{align}
&-\int_M\sum_{i=1}^m\langle R^N(\overline{\Delta}\tau(\varphi),d\varphi(e_i))d\varphi(e_i),\overline{\Delta}\tau(\varphi)\rangle\,\eta^2\,v_g \nonumber\\
&=\int_M\eta^2\,\sum_{i=1}^m
\langle R^N(d\varphi(e_i),\overline{\Delta}\tau(\varphi))d\varphi(e_i),\overline{\Delta}\tau(\varphi)\rangle\,v_g\nonumber\\
&=c\int_M\eta^2\,\sum_{i=1}^m\bigg\{
\langle \overline{\Delta}\tau(\varphi),d\varphi(e_i)\rangle^2-
\langle d\varphi(e_i),d\varphi(e_i)\rangle\,\langle\overline{\Delta}\tau(\varphi),\overline{\Delta}\tau(\varphi)\rangle
\bigg\}\,v_g\nonumber\\
&=c\int_M\eta^2\,\bigg\{
\langle \overline{\Delta}\tau(\varphi),d\varphi\rangle^2-\vert\overline{\Delta}\tau(\varphi)\vert^2\,\vert d\varphi\vert^2\bigg\}\,v_g\nonumber\\
&\geq 0
\end{align}
since $c\leq 0$. 
\vskip0.6cm\par
({\it The fourth step})\quad For the third term of the left hand side of (3.3), 
since 
$$\langle \tau(\varphi),\overline{\Delta}\tau(\varphi)\rangle=\frac12\Delta\,\vert\tau(\varphi)\vert^2+\vert\overline{\nabla}\tau(\varphi)\vert^2,$$ 
and 
$$\vert\tau(\varphi)\vert^2\,\Delta\,\vert\tau(\varphi)\vert^2=\frac12\,\Delta\,\vert\tau(\varphi)\vert^4+\vert\,\nabla\,\vert\tau(\varphi)\vert^2\,\vert^2,$$ 
we have 
\begin{align}
-c\int_M&\langle\tau(\varphi),\tau(\varphi)\rangle\,\langle\tau(\varphi),\overline{\Delta}\tau(\varphi)\rangle\,\eta^2\,v_g\nonumber\\
&=
-\frac{c}{2}\,\int_M\vert\tau(\varphi)\vert^2\,\Delta\,\vert\tau(\varphi)\vert^2\,\eta^2\,v_g\nonumber\\
&\quad -c\,\int_M\vert\tau(\varphi)\vert^2\,\vert\overline{\nabla}\tau(\varphi)\vert^2\,\eta^2\,v_g\nonumber\\
&=-\frac{c}{4}\,\int_M\Delta\,\vert\tau(\varphi)\vert^4\,\eta^2\,v_g-\frac{c}{2}\,\int_M\vert\,\nabla\,\vert\tau(\varphi)\vert^2\,\vert^2\,\eta^2\,v_g\nonumber\\
&\quad -c\,\int_M\vert\tau(\varphi)\vert^2\,\vert\overline{\nabla}\tau(\varphi)\vert^2\,\eta^2\,v_g\nonumber\\
&=-\frac{c}{4}\,\int_M\sum_{i=1}^m
\langle\nabla_{e_i}\,\vert\tau(\varphi)\vert^4,\nabla_{e_i}\eta^2\rangle\,\,v_g
-\frac{c}{2}\,\int_M\vert\,\nabla\,\vert\tau(\varphi)\vert^2\,\vert^2\,\eta^2\,v_g
\nonumber\\
&\quad -c\,\int_M\vert\tau(\varphi)\vert^2\,\vert\overline{\nabla}\tau(\varphi)\vert^2\,\eta^2\,v_g\nonumber\\
&=-\frac{c}{2}\,\int_M\eta\,\sum_{i=1}^m
\langle\nabla_{e_i}\,\vert\tau(\varphi)\vert^4,\nabla_{e_i}\eta\rangle\,\,v_g
-\frac{c}{2}\,\int_M\vert\,\nabla\,\vert\tau(\varphi)\vert^2\,\vert^2\,\eta^2\,v_g
\nonumber\\
&\quad -c\,\int_M\vert\tau(\varphi)\vert^2\,\vert\overline{\nabla}\tau(\varphi)\vert^2\,\eta^2\,v_g.
\end{align}
\vskip0.6cm\par
({\it The fifth step})\quad Here, we have 
\begin{align}
\sum_{i=1}^m\vert\,\langle\nabla_{e_i}\,\vert\tau(\varphi)\vert^4,\nabla_{e_i}\eta\rangle\,\vert
&\leq \sum_{i=1}^m\vert\,\nabla_{e_i}\vert\tau(\varphi)\vert^4\,\vert\,
\vert\nabla\eta\vert\nonumber\\
&=\vert\nabla\,\vert\tau(\varphi)\vert^4\vert\,\vert\nabla\eta\vert,  
\end{align}
so that we have, since $c\leq 0$, 
\begin{align}
\mbox{(3.8)}&\geq \frac{c}{2}\,\int_M\vert\,\nabla\,\vert\tau(\varphi)\vert^4\,\vert\,\,\eta\,\,\vert\nabla\eta\vert\,v_g\nonumber\\
&\qquad-\frac{c}{2}\,\int_M\vert\,\nabla\,\vert\tau(\varphi)\vert^2\,\vert^2\,\eta^2\,v_g\nonumber\\
&\qquad-c\,\int_M\vert\tau(\varphi)\vert^2\,\vert\overline{\nabla}\,\tau(\varphi)\vert^2\,\eta^2\,v_g. 
\end{align}
Since we have 
$$
\vert\,\nabla\,\vert\tau(\varphi)\vert^4\,\vert=2\,\vert\tau(\varphi)\vert^2\,\cdot\,\vert\,\nabla\,\vert\tau(\varphi)\vert^2\,\vert,
$$
the right hand side of (3.10) is equal to 
\begin{align}
&c\,\int_M\vert\,\nabla\,\vert\tau(\varphi)\vert^2\,\vert\,\eta\,\cdot\,\vert\tau(\varphi)\vert^2\,\vert\nabla\eta\vert\,v_g\nonumber\\
&\quad 
-\frac{c}{2}\,\int_M\vert\,\nabla\,\vert\tau(\varphi)\vert^2\,\vert^2\,\eta^2\,v_g
\nonumber\\
&\quad-c\,\int_M\vert\tau(\varphi)\vert^2\,\vert\overline{\nabla}\,\tau(\varphi)\vert^2\,\eta^2\,v_g. 
\end{align}
Now applying again for 
$A:=\vert\,\nabla\,\vert\tau(\varphi)\vert^2\,\vert\,\eta$, and 
$B:=\vert\tau(\varphi)\vert^2\,\vert\nabla\eta\vert$, 
the Young's inequality: for every positive number $\delta>0$, 
$$
\pm2\,\langle A,B\rangle\leq \delta\,\vert A\vert^2+\frac{1}{\delta}\,\vert B\vert^2,
$$ 
we obtain because of $c\leq 0$, 
\begin{align}
\mbox{(3.11)}&\geq 
\frac{c\,\delta}{2}\,\int_M\vert\,\nabla\,\vert\tau(\varphi)\vert^2\,\vert^2\,\eta^2\,v_g
+\frac{c}{2\delta}\,\int_M\vert\tau(\varphi)\vert^4\,\,\vert\nabla\eta\vert^2\,v_g\nonumber\\
&\qquad-\frac{c}{2}\,\int_M\,\vert\nabla\,\vert\tau(\varphi)\vert^2\,\vert^2\,\,\eta^2\,v_g\nonumber\\
&\qquad-c\,\int_M\vert\tau(\varphi)\vert^2\,\vert\overline{\nabla}\tau(\varphi)\,\vert^2\,\eta^2\,v_g\nonumber\\
&=-\frac{c}{2}(1-\delta)\,\int_M\vert\,\nabla\,\vert\tau(\varphi)\vert^2\,\vert^2\,\eta^2\,v_g\nonumber\\
&\qquad+\frac{c}{2\delta}\,\int_M\vert\tau(\varphi)\vert^4\,\,\vert\nabla\eta\vert^2\,v_g\nonumber\\
&\qquad-c\int_M\vert\tau(\varphi)\vert^2\,\,\vert\overline{\nabla}\,\tau(\varphi)\vert^2\,\eta^2\,v_g. 
\end{align}
By putting $\delta=\frac12$, we obtain 
\begin{align}
\mbox{(3.12)}&=-\frac{c}{4}\,\int_M
\vert\,\nabla\,\vert\tau(\varphi)\vert^2\,\vert^2\,\eta^2\,v_g\nonumber\\
&\qquad
+c\,\int_M\vert\tau(\varphi)\vert^4\,\,\vert\nabla\eta\vert^2\,v_g\nonumber\\
&\qquad
-c\int_M\vert\tau(\varphi)\vert^2\,\,\vert\overline{\nabla}\,\tau(\varphi)\vert^2\,\eta^2\,v_g. 
\end{align}
\vskip0.6cm\par
({\it The sixth step}) \quad 
All together the above, we obtain 
\begin{align}
0&\geq 
(1-\epsilon)
 \int_M\vert\overline{\nabla}\,\overline{\Delta}\tau(\varphi)\vert^2\,\eta^2\,v_g
 -\frac{1}{\epsilon}\int_M 
 \vert\overline{\Delta}\tau(\varphi)\vert^2\,\vert\nabla\eta\vert^2\,v_g\nonumber\\
 &\qquad
 -\frac{c}{4}\,\int_M
\vert\,\nabla\,\vert\tau(\varphi)\vert^2\,\vert^2\,\eta^2\,v_g
+c\,\int_M\vert\tau(\varphi)\vert^4\,\,\vert\nabla\eta\vert^2\,v_g\nonumber\\
&\qquad
-c\int_M\vert\tau(\varphi)\vert^2\,\,\vert\overline{\nabla}\,\tau(\varphi)\vert^2\,\eta^2\,v_g
\end{align}
which is equivalent to that 
\begin{align}
&\frac{1}{\epsilon}\int_M 
 \vert\overline{\Delta}\tau(\varphi)\vert^2\,\vert\nabla\eta\vert^2\,v_g
 -c\,\int_M\vert\tau(\varphi)\vert^4\,\,\vert\nabla\eta\vert^2\,v_g\nonumber\\
 &\geq 
 (1-\epsilon)
 \int_M\vert\overline{\nabla}\,\overline{\Delta}\tau(\varphi)\vert^2\,\eta^2\,v_g
 -\frac{c}{4}\,\int_M
\vert\,\nabla\,\vert\tau(\varphi)\vert^2\,\vert^2\,\eta^2\,v_g\nonumber\\
&\qquad
-c\int_M\vert\tau(\varphi)\vert^2\,\,\vert\overline{\nabla}\,\tau(\varphi)\vert^2\,\eta^2\,v_g.
\end{align}
Here, we put $\epsilon=\frac12$ in (3.15), and notice that 
$\eta=1$ on $B_r(x_0)$, and $\vert\nabla\eta\vert\leq \frac{2}{r}$. Then, we have 
\begin{align}
&\frac{8}{r^2}\int_M 
 \vert\overline{\Delta}\tau(\varphi)\vert^2\,v_g
 -\frac{4c}{r^2}\,\int_M\vert\tau(\varphi)\vert^4\,v_g\nonumber\\
 &\geq 
 \frac12
 \int_{B_r(x_0)}\vert\overline{\nabla}\,\overline{\Delta}\tau(\varphi)\vert^2\,v_g
 -\frac{c}{4}\,\int_{B_r(x_0)}
\vert\,\nabla\,\vert\tau(\varphi)\vert^2\,\vert^2\,v_g\nonumber\\
&\qquad
-c\int_{B_r(x_0)}\vert\tau(\varphi)\vert^2\,\,\vert\overline{\nabla}\,\tau(\varphi)\vert^2\,v_g.
\end{align}
\vskip0.6cm\par
({\it The seventh step})\quad 
By virtue of our assumptions that 
$\int_M\vert \overline{\Delta}\tau(\varphi)\vert^2\,v_g<\infty$ and 
$\int_M\vert\tau(\varphi)\vert^4\,v_g<\infty$, 
and $B_r(x_0)$ goes to $M$ if $r\rightarrow \infty$ because of completeness of $(M,g)$, the left hand side of (3.16) goes to zero if $r\rightarrow\infty$. We obtain 
\begin{align}
&\frac12
 \int_M\vert\overline{\nabla}\,\overline{\Delta}\tau(\varphi)\vert^2\,v_g
 -\frac{c}{4}\,\int_{M}
\vert\,\nabla\,\vert\tau(\varphi)\vert^2\,\vert^2\,v_g\nonumber\\
&-c\int_{M}\vert\tau(\varphi)\vert^2\,\,\vert\overline{\nabla}\,\tau(\varphi)\vert^2\,v_g\nonumber\\
&\leq 0. 
\end{align}
Since $c\leq0$, all the terms of (3.17) are non-negative and we have 
\begin{equation}
\overline{\nabla}\,\overline{\Delta}\tau(\varphi)=0. 
\end{equation}
In the case $c<0$, we have 
\begin{equation}
\left\{
\begin{aligned}
&\overline{\nabla}\,\overline{\Delta}\tau(\varphi)=0,\\
&\nabla\,\vert\tau(\varphi)\vert^2=0,\\
&\vert\tau(\varphi)\vert^2\,\vert\overline{\nabla}\tau(\varphi)\vert^2=0.
\end{aligned}
\right.
\end{equation}
Notice that by (3.18), $\vert\overline{\Delta}\tau(\varphi)\vert^2$ is constant, say $c_0$.
Because, for every $C^{\infty}$ vector field $X$ on $M$, by (3.18), 
$$X\,\vert \overline{\Delta}\tau(\varphi)\vert^2=2\,\langle\overline{\nabla}_X\,\overline{\Delta}\tau(\varphi),\overline{\Delta}\tau(\varphi)\rangle=0.$$ 
\vskip0.2cm\par
({\it The eighth step})\quad 
In the case that $c<0$, 
by the second equation of (3.19), $\vert \tau(\varphi)\vert^2$ is constant. 
Therefore, by the last equation of (3.19), 
it holds that 
$\vert \tau(\varphi)\vert\equiv 0$ or $\vert \overline{\nabla}\tau(\varphi)\vert^2=0$, 
i.e., $\overline{\nabla}\tau(\varphi)\equiv0$. We have 
$\overline{\Delta}\tau(\varphi)\equiv0$. 
Since $\varphi:\,(M,g)\rightarrow N(c)$ is triharmonic, (3.2) holds. 
Substituting $\overline{\Delta}\tau(\varphi)\equiv0$ in (3.2), 
we obtain $\tau(\varphi)\equiv0$ in the case of $c<0$. 
\par
In the case that $c=0$ and $\mbox{\rm Vol}(M,g)=\infty$, we have 
$$\infty>\widetilde{E}_4(\varphi)=E_4(\varphi)=\frac12\int_M\vert\overline{\Delta}\tau(\varphi)\vert^2\,v_g=c_0\,\mbox{\rm Vol}(M,g).$$
Thus, we obtain $c_0=0$, i.e., $\overline{\Delta}\tau(\varphi)\equiv0$, i.e., $\varphi$ is a biharmonic map of $(M,g)$ into the Euclidean space $N(0)$. Then, applying 
(2) in Theorem 2.1 of \cite{NU3}, or 
Theorem 3.1 in \cite{NU4}, $\varphi$ is harmonic, i.e., minimal, by virtue of the assumption that 
the bienergy $E_2(\varphi)=\frac12\int_M\vert \tau(\varphi)\vert^2\,$ is finite. 
\par
In the case that $c=0$ and $\mbox{\rm Vol}(M,g)<\infty$, we have  
$E(\varphi)=\frac{m}{2}\,\mbox{\rm Vol}(M,g)<\infty$ since $\varphi$ is an isometric immersion. 
Furthermore, we have that
$$
E_2(\varphi)=\frac12\int_M\vert\tau(\varphi)\vert^2\,v_g\leq 
\frac12\left(
\int_M1\,v_g\right)^{1/2}\,\left(\int_M
\vert\tau(\varphi)\vert^4\,v_g
\right)^{1/2}<\infty.
$$  
By virtue of the assumption that the 4-energy $E_4(\varphi)$ and 
the 3-energy 
$E_3(\varphi)=\frac12\int_M\vert\overline{\nabla}\tau(\varphi)\vert^2\,v_g$ are finite, 
we can apply again Theorem 3.1 in \cite{NU4}, 
and then we also obtain 
$\tau(\varphi)=0$, i.e., $\varphi$ is minimal. 
\par
We have Theorem 2.4. 
\end{proof}

\vskip0.6cm\par
\section{Triharmonic isometric immersions with the constant mean curvature}
In the case that $\vert \tau(\varphi)\vert $ is constant, that is, the mean curvature is constant, 
the finiteness of $\int_M\vert 
\overline{\Delta}\tau(\varphi)\vert^2\,v_g$ 
in Theorem 2.4 can be replaced into 
the weaker condition as follows.

\begin{thm}\label{sub Th}
Let $\varphi:\,(M,g)\rightarrow N(c)$ be an isometric immersion with the constant mean curvature from a complete Riemannian manifold $(M,g)$ into 
a Riemannian manifold $N(c)$ of negatively constant curvature $c$. 
If $\varphi$ is triharmonic and $\frac12\int_M\vert \overline{\Delta}\tau(\varphi)\vert^p\,v_g<\infty$  $(\text{for some}\ 2 \leq p<\infty)$, 
then $\varphi$ 
is harmonic, i.e., minimal. 
\end{thm}

Before mentioning the proof of Theorem $\ref{sub Th}$, 
we shall show the following lemma.

\begin{lem}\label{4.key lemma}
Assume that 
$\alpha\in \Gamma(\varphi^{-1}TN)$ 
satisfies that 
$$
\vert\alpha\vert^q\,\,\vert \overline{\nabla}\alpha\vert^2=0
\qquad (\mbox{for some $q>0$}).
$$
Then, $(1)$
$\vert \alpha\vert$ is constant everywhere on $M$,
and then \par\quad \,\,\,
$(2)$ 
either $\alpha=0$ or $\overline{\nabla}\alpha=0$.  
\end{lem}
\begin{proof}
We first notice that for every $\alpha\in \Gamma(\varphi^{-1}TN)$, it holds that 
\begin{align}\label{4.1}
\vert \,\nabla\,\vert\alpha\vert\,\vert\leq \vert\overline{\nabla}\alpha\vert
\qquad (\mbox{everywhere on the set $\{x\in M\vert\,\vert \alpha_x\not=0\}$}).
\end{align} 
Because 
\begin{align}\label{4.2}
\vert\alpha\vert\,\vert\nabla\,\vert\alpha\vert\,\vert&=\frac12\,\vert\nabla\,\vert\alpha\vert^2\,\vert\nonumber\\
&=\frac12\,\vert\,\nabla\,h(\alpha,\alpha)\vert\nonumber\\
&=\vert h(\overline{\nabla}\alpha,\alpha)\vert\nonumber\\
&\leq \vert\,\overline{\nabla}\alpha\vert\,\vert\alpha\vert,
\end{align}
so we obtain $(\ref{4.1})$ due to $(\ref{4.2})$. 
\medskip\par
Therefore, we have 
\begin{align}
0\leq \vert\alpha\vert^q\,\vert\nabla\,\vert\alpha\vert\,\vert^2\leq \vert\alpha\vert^q\,\vert\overline{\nabla}\alpha\vert^2=0
\end{align}
everywhere on $M$. 
Thus, we have 
\begin{align}\label{4.4}
\vert\alpha\vert^q\,\vert\,\nabla\,\vert\alpha\vert\,\vert^2=0. 
\end{align}
By $(\ref{4.4})$, 
we have 
\begin{align}
\bigg(\frac{2}{q+2}
\bigg)^2\,\big\vert
\nabla\,\vert\alpha\vert^{\frac{q}{2}+1}
\big\vert^2
&=\bigg(\frac{2}{q+2}
\bigg)^2\,\bigg\vert
\bigg(\frac{q}{2}+1
\bigg)\,\vert\alpha\vert^{q/2}\,\nabla\,\vert\alpha\vert
\bigg\vert^2\nonumber\\
&=\vert\alpha\vert^q\,\vert\nabla\,\vert\alpha\vert\,\vert^2\nonumber\\
&=0.
\end{align}
We have 
\begin{align}
\nabla\,\vert\alpha\vert^{\frac{q}{2}+1}=0, 
\end{align}
which implies that 
$\vert\alpha\vert^{q/2+1}$ is constant, i.e., $\vert\alpha\vert$ is a constant, say $C_0$. 
Then, 
\par
$(1)$ in the case that $C_0=0$, we have $\alpha=0$. 
$(2)$ In the case that $C_0\not=0$, 
we have 
\begin{align}
C_0{}^q\,\vert\overline{\nabla}\alpha\vert^2&=\vert\alpha\vert^q\,\vert\overline{\nabla}\alpha\vert^2\nonumber\\
&=0
\end{align}
by virtue of the assumption of Lemma $\ref{4.key lemma}$. We obtain 
$\overline{\nabla}\alpha=0$.
\end{proof}

By using Lemma $\ref{4.key lemma}$, we shall show Theorem $\ref{sub Th}$.

\begin{proof}[Proof of Theorem $\ref{sub Th}$]
We will use an argument similar to the proof of Theorem~$\ref{main Th}$.
\par
We take the cut-off function in the first step of the proof of Theorem $\ref{main Th}$.
By taking the inner product of (3.2) and $\vert \overline{\Delta}\tau(\varphi)\vert^{p-2}\, \overline{\Delta}\tau(\varphi)\,\eta^2$ 
and integrate over $M$, so we have 
\begin{align}\label{1}
&\int_M\langle\overline{\Delta}^2(\tau(\varphi)),\,\vert \overline{\Delta}\tau(\varphi)\vert^{p-2}\,\overline{\Delta}\tau(\varphi)\,\eta^2\rangle\,v_g\nonumber\\
&\quad-\int_M
\sum_{i=1}^m
\langle R^N(\overline{\Delta}\tau(\varphi),d\varphi(e_i))d\varphi(e_i),\vert \overline{\Delta}\tau(\varphi)\vert^{p-2}\,\overline{\Delta}\tau(\varphi)\rangle\,\eta^2\,v_g\nonumber\\
&\quad-c\,\int_M\langle\tau(\varphi),\tau(\varphi)\rangle\,\langle \tau(\varphi),\vert \overline{\Delta}\tau(\varphi)\vert^{p-2}\,\overline{\Delta}\tau(\varphi)\rangle\,\eta^2\,v_g\nonumber\\
&\quad= 0.
\end{align}
\vskip0.6cm\par
For the first term of the left hand side of $(\ref{1})$,  
\begin{align}\label{2}
&\int_M\langle\overline{\Delta}^2(\tau(\varphi)),\,\vert \overline{\Delta}\tau(\varphi)\vert^{p-2}\,\overline{\Delta}\tau(\varphi)\,\eta^2\rangle\,v_g\notag\\
&\qquad\qquad=\int_M\sum_{i=1}^m\langle 
\overline{\nabla}_{e_i}(\overline{\Delta}\tau(\varphi)),\overline{\nabla}_{e_i}(\vert \overline{\Delta}\tau(\varphi)\vert^{p-2}\,\overline{\Delta}\tau(\varphi)\,\eta^2)\rangle\,v_g.
\end{align}
Here we use 
$$
\nabla_{e_i} \vert \overline{\Delta}\tau(\varphi)\vert^{p-2}
=(p-2)\vert \overline{\Delta}\tau(\varphi)\vert^{p-4} \langle \overline{\nabla}_{e_i}\overline{\Delta}\tau(\varphi),\overline{\Delta}\tau(\varphi)\rangle,
$$
 and then the right hand side of $(\ref{2})$ coincides with 
\begin{align}\label{3}
\int_M
&\vert \overline{\Delta}\tau(\varphi)\vert^{p-2}\,\vert\overline{\nabla}\,\overline{\Delta}\tau(\varphi)\vert^2\,\eta^2\,v_g\nonumber\\
&+2\int_M\sum_{i=1}^m \langle 
\eta\,\vert \overline{\Delta}\tau(\varphi)\vert^{\frac{p}{2}-1}\,\overline{\nabla}_{e_i}(\overline{\Delta}\tau(\varphi)),\nabla_{e_i}\eta\,\vert \overline{\Delta}\tau(\varphi)\vert^{\frac{p}{2}-1}\,\overline{\Delta}\tau(\varphi)\rangle\,v_g\nonumber\\
&+(p-2)\int_M\sum_{i=1}^m \vert \overline{\Delta}\tau(\varphi)\vert^{p-4}
\,\langle 
\overline{\nabla}_{e_i}(\overline{\Delta}\tau(\varphi)),\overline{\Delta}\tau(\varphi)\rangle^2\,\eta^2v_g.
\end{align}
For the second term of $(\ref{3})$, 
put 
$S_i:=\eta\,\vert \overline{\Delta}\tau(\varphi)\vert^{\frac{p}{2}-1}\,\overline{\nabla}_{e_i}(\overline{\Delta}\tau(\varphi))$, and 
$T_i:=\nabla_{e_i}\eta\,\vert \overline{\Delta}\tau(\varphi)\vert^{\frac{p}{2}-1}\,\overline{\Delta}\tau(\varphi)$ ($i=1\,\cdots,m$), and 
recall the Young's inequality: 
for every $\epsilon>0$, 
$$
\pm2\langle S_i,T_i\rangle\leq \epsilon\,\vert S_i\vert^2+\frac{1}{\epsilon}\,\vert T_i\vert^2,
$$    
because of the inequality 
 $
 0\leq \vert \sqrt{\epsilon}\,S_i\pm\frac{1}{\sqrt{\epsilon}}\,T_i\vert^2.  
 $
 Therefore, $(\ref{3})$ is bigger than or equal to
 \begin{align}\label{key i-1}
 \int_M&\vert \overline{\Delta}\tau(\varphi)\vert^{p-2} \vert\overline{\nabla}\,\overline{\Delta}\tau(\varphi)\vert^2\,\eta^2\,v_g\nonumber\\
 &\qquad\qquad-\bigg\{
 \epsilon\,\int_M\eta^2\,\vert \overline{\Delta}\tau(\varphi)\vert^{p-2}\vert \overline{\nabla}\,\overline{\Delta}\tau(\varphi)\vert^2\,v_g
 +\frac{1}{\epsilon}\,\int_M 
 \vert\nabla\eta\vert^2\,\vert\overline{\Delta}\tau(\varphi)\vert^p\,v_g
 \bigg\}\nonumber\\
 &=(1-\epsilon)
 \int_M\eta^2\,\vert \overline{\Delta}\tau(\varphi)\vert^{p-2}\vert \overline{\nabla}\,\overline{\Delta}\tau(\varphi)\vert^2\,v_g
 -\frac{1}{\epsilon}\,\int_M 
 \vert\nabla\eta\vert^2\,\vert\overline{\Delta}\tau(\varphi)\vert^p\,v_g.
 \end{align}
\vskip0.6cm\par
 For the second term of the left hand side of $(\ref{1})$, by the same reason of the third step of Theorem $\ref{main Th}$, we have
\begin{align}\label{key i-2}
&-\int_M\sum_{i=1}^m\langle R^N(\overline{\Delta}\tau(\varphi),d\varphi(e_i))d\varphi(e_i),\overline{\Delta}\tau(\varphi) \rangle \vert \overline{\Delta}\tau(\varphi)\vert^{p-2}\,\eta^2\geq 0.
\end{align}
\vskip0.6cm\par
For the third term of the left hand side of $(\ref{1})$, 
since the mean curvature is constant, 
$$\langle \tau(\varphi),\overline{\Delta}\tau(\varphi)\rangle=\vert\overline{\nabla}\tau(\varphi)\vert^2.$$ 
By using this, the third term of the left hand side of $(\ref{1})$ is equal to 
\begin{align}\label{key i-3}
&-c\,\int_M\langle\tau(\varphi),\tau(\varphi)\rangle\,\langle \tau(\varphi),\overline{\Delta}\tau(\varphi)\rangle\,\vert \overline{\Delta}\tau(\varphi)\vert^{p-2}\,\eta^2\,v_g\notag\\
&\qquad=-c\,\int_M\vert \overline{\Delta}\tau(\varphi)\vert^{p-2} \vert\tau(\varphi)\vert^2\,\vert\overline{\nabla}\tau(\varphi)\vert^2\,\eta^2\,v_g.
\end{align}
Combining $(\ref{key i-1})$, $(\ref{key i-2})$ and $(\ref{key i-3})$, and noticing that $\eta=1$ on $B_{r}(x_0)$, we have
\begin{align}\label{key i last}
&(1-\epsilon)
 \int_{B_{r}(x_0)}\eta^2\,\vert \overline{\Delta}\tau(\varphi)\vert^{p-2}\vert \overline{\nabla}\,\overline{\Delta}\tau(\varphi)\vert^2\,v_g\notag\\
 &\quad-c\,\int_{B_{r}(x_0)}\vert \overline{\Delta}\tau(\varphi)\vert^{p-2} \vert\tau(\varphi)\vert^2\,\vert\overline{\nabla}\tau(\varphi)\vert^2\,\eta^2\,v_g\notag\\
 &\quad\leq \frac{1}{\epsilon}\,\int_M 
 \vert\nabla\eta\vert^2\,\vert\overline{\Delta}\tau(\varphi)\vert^p\,v_g\notag\\
  &\quad\leq \frac{1}{\epsilon}\frac{4}{r^2}\,\int_M 
\vert\overline{\Delta}\tau(\varphi)\vert^p\,v_g,
\end{align}
where the last inequality follows from $\vert \nabla \eta\vert \leq \frac{2}{r}$.
By virtue of our assumptions that 
$\int_M\vert \overline{\Delta}\tau(\varphi)\vert^p\,v_g<\infty$ 
and $B_r(x_0)$ goes to $M$ if $r\rightarrow \infty$ because of completeness of $(M,g)$, the right hand side of $(\ref{key i last})$ goes to zero if $r\rightarrow\infty$. We obtain 
\begin{align}\label{key i last 2}
(1-\epsilon)
 \int_M \vert \overline{\Delta}\tau(\varphi)\vert^{p-2}\vert \overline{\nabla}\,\overline{\Delta}\tau(\varphi)\vert^2\,v_g
 - c\,\int_M\vert \overline{\Delta}\tau(\varphi)\vert^{p-2} \vert\tau(\varphi)\vert^2\,\vert\overline{\nabla}\tau(\varphi)\vert^2\,v_g \leq 0.
\end{align}
Since we can take that $0<\epsilon<1$, and the assumption $c<0$, all the terms of $(\ref{key i last 2})$ are non-negative and we have 
\begin{equation}\label{c2}
\left\{
\begin{aligned}
&\vert \overline{\Delta}\tau(\varphi)\vert^{p-2}\vert \overline{\nabla}\,\overline{\Delta}\tau(\varphi)\vert^2=0,\\
&\vert \overline{\Delta}\tau(\varphi)\vert^{p-2} \vert\tau(\varphi)\vert^2\,\vert\overline{\nabla}\tau(\varphi)\vert^2=0.
\end{aligned}
\right.
\end{equation}

 If we put $\alpha:=\overline{\Delta}\,\tau(\varphi)$, 
 $(\ref{c2})$ is equivalent to that 
 \begin{align}\label{c21}
&
 \vert\alpha\vert^{p-2}\,\vert\overline{\nabla}\alpha\vert^2=0,
 \end{align}
and
\begin{align}\label{c22}
 \vert\alpha\vert^{p-2}\,\vert\tau(\varphi)\vert^2\,\vert\overline{\nabla}\tau(\varphi)\vert^2=0.
\end{align}
 Applying Lemma $\ref{4.key lemma}$, for putting $q:=p-2>0$, by $(\ref{c21})$, 
 we have 
 $\vert\overline{\Delta}\,\tau(\varphi)\vert =\vert\alpha\vert$ is a constant, say $C_1$. 
 \par
 In the Case (I): $C_1=0$, we have 
 $$
 \overline{\Delta}\,\tau(\varphi)=\alpha=0\qquad (\mbox{everywhere on }M.) 
 $$
 By vanishing of the tritension field (3.2), we have $\tau(\varphi)=0$ on $M$. 
 \par
 In the Case (II): $C_1\not=0$, we have by $(\ref{c22})$, 
 \begin{align}\label{c23}
 \vert\tau(\varphi)\vert^2\,\vert\overline{\nabla}\,\tau(\varphi)\vert^2=0.
\end{align}
 By applying Lemma $\ref{4.key lemma}$ for $\alpha:=\tau(\varphi)$ and $q=2$, 
 we have 
 $\vert\tau(\varphi)\vert$ is a constant, say $C_2$. 
 In the case (II-i) $C_2=0$, we clearly have $\tau(\varphi)=0$ on $M$. 
 In the case (II-ii) $C_2\not=0$, 
 we have $\overline{\nabla}\,\tau(\varphi)=0$ on $M$ by virtue of $(\ref{c23})$. 
   Then, we have $\overline{\Delta}\tau(\varphi)=0$ on $M$ which contradicts that $C_1\not=0$. 
   This case (II-ii) does not occur. 
\end{proof}

\begin{rem}
 (1) Let $\varphi:(M,g)\rightarrow N(0)=\mathbb{E}^n$ be an isometric immersion of a complete Riemannian manifold $(M,g)$ into the Euclidean space $N(0)$ with $\vert \tau(\varphi) \vert$ is constant. 
 In the case that $\mbox{\rm Vol}(M,g)<\infty$, if $\varphi$ is triharmonic and both $\int_M\vert \tau(\p)\vert^p\,v_g<\infty$ (for some $2\leq p<\infty$) and the 3-energy is finite, then $\varphi$ is harmonic, i.e., minimal (cf. \cite{SM1}).

(2) In the case that ${\rm Vol}(M,g)=\infty$,  
if $\varphi:(M,g)\rightarrow (N,h)$ is a smooth map from a Riemannian manifold $(M,g)$ into a Riemannian manifold $(N,h)$ with $\vert \tau(\varphi) \vert$ is constant and 
 $\int_{M}\vert \tau(\varphi)\vert^p\, v_g<\infty$ (for some $0< p<\infty$), then $\varphi$ is harmonic.
\par
\end{rem}

\vskip0.6cm\par







\end{document}